\documentclass[11pt,english,oneside]{amsart}

\usepackage[breaklinks,colorlinks,linkcolor=blue,citecolor=red,urlcolor=black]{hyperref}
\usepackage{enumerate,empheq,geometry,amsmath,graphicx}
\usepackage[utf8]{inputenc}

\theoremstyle{plain}
\newtheorem{theorem}{Theorem}[section]

\newtheorem{proposition}[theorem]{Proposition}
\newtheorem{remark}[theorem]{Remark}

\newtheorem{lemma}[theorem]{Lemma}
\newtheorem{definition}[theorem]{Definition}
\newcommand{\dd}{\mathrm{d}}

\title[Minimal hypersurfaces in spheres generated by isoparametric foliations
]
{Minimal hypersurfaces in spheres generated by isoparametric foliations
}
\author{Junqi Lai and  Guoxin Wei}
\address{Junqi Lai \\  School of Mathematical Sciences, South China Normal University,
510631, Guangzhou,  China, 2019021668@m.scnu.edu.cn}
\address{Guoxin Wei \\  School of Mathematical Sciences, South China Normal University,
510631, Guangzhou,  China, weiguoxin@tsinghua.org.cn}

\begin{document}

	\begin{abstract}
    We investigate the existence of minimal hypersurfaces in $\mathbb{S}^{n+1}$ that are generated by the isoparametric foliation of a subsphere $\mathbb{S}^n$. By considering a generalized rotational ansatz formed by the union of homothetic copies of isoparametric leaves, we reduce the minimal surface equation to an ordinary differential equation. We prove that this construction yields a closed embedded minimal hypersurface for any choice of isoparametric hypersurface $M \subset \mathbb{S}^n$. The resulting hypersurfaces have the topological type $S^1 \times M$, extending the known examples of minimal hypertori ($S^1\times S^k\times S^k$ and $S^1\times S^k\times S^l$) to a broader class of topologies determined by isoparametric structures.
	\end{abstract}
	\maketitle
	
	\section{Introduction}
The study of minimal hypersurfaces in the unit sphere $\mathbb{S}^{n+1}$ has long been a central topic in differential geometry. Since the foundational works of Lawson \cite{Lawson1970} and Chern, do Carmo, and Kobayashi \cite{Chern1970}, understanding the classification, topology, and construction of closed minimal hypersurfaces has attracted significant attention. While the equator and the Clifford tori provide the simplest examples, finding new embedded minimal hypersurfaces with non-trivial topology remains a challenging and fruitful area of research.

A particularly rich source of geometric structures in spheres comes from the theory of isoparametric hypersurfaces. Initiated by Cartan and structurally clarified by M\"unzner \cite{Munzner1980}, this theory describes hypersurfaces with constant principal curvatures. The level sets of the defining isoparametric polynomial form a singular Riemannian foliation of $\mathbb{S}^n$. The geometry of these leaves is well-understood and highly symmetric, making them ideal candidates for constructing complex submanifolds in higher dimensions.

In this paper, we explore the connection between these two classical subjects: isoparametric foliations in $\mathbb{S}^n$ and minimal hypersurfaces in $\mathbb{S}^{n+1}$. A natural question is whether one can "lift" the geometry of an isoparametric hypersurface $M \subset \mathbb{S}^n$ to construct a minimal hypersurface in the higher-dimensional sphere $\mathbb{S}^{n+1}$. Specifically, we investigate a construction based on the union of homothetic copies of the isoparametric leaves. This can be viewed as a generalization of the rotational surface construction.

Our main contribution is to show that such a construction is always possible for any isoparametric hypersurface, yielding a new family of minimal hypersurfaces with the topological type of a product $S^1 \times M$. The main result is stated as follows:

\begin{theorem}\label{MainTheorem}
    For any isoparametric hypersurface \( M \) in \( \mathbb{S}^n\), there is a closed embedded minimal hypersurface of topological type \( S^1 \times M \) in \( \mathbb{S}^{n+1}\). This hypersurface is a union of homothetic copies of the leaves of the isoparametric foliation of \( \mathbb{S}^n\) associated to \( M \).
\end{theorem}

The proof is inspired by the ideas in~\cite{R}.
Geometrically, this hypersurface fibers over a closed curve in the parameter space, with fibers being the leaves of the isoparametric foliation. This construction provides a uniform way to produce minimal hypersurfaces in $\mathbb{S}^{n+1}$ reflecting the complexity of the isoparametric theory in $\mathbb{S}^n$.

\begin{remark}
	In the case where the isoparametric hypersurface \(M\) is \(S^k\times S^k\), our main theorem recovers Theorem 1.1 of \cite{CarS}, which was proved using ODE method.
	For \(M=S^k\times S^l\), it recovers that of \cite{FT}, which was established by doubling the links of free-boundary minimal cones in Euclidean space with bi-orthogonal symmetry.
	For more recent development of construction of minimal and constant mean curvature hypersurfaces in spheres, one can see \cite{HW,LW,P} (using ODE method) and \cite{WWZ,WZ} (using min-max constructions).
\end{remark}


	\section{Preliminaries}\label{prel}
	Let \(M^{n-1}\) be an isoparametric hypersurface in the unit sphere \(\mathbb{S}^n \subset \mathbb{R}^{n+1}\), and let \(N\) be a unit normal vector field to \(M\) in \(\mathbb{S}^n\). It is well known that the principal curvatures of \(M\) are constant. Let \(g\) denote the number of distinct principal curvatures of \(M\).
	A celebrated result by M\"unzner states that \(g\) can only take values in $\left\{ 1, 2, 3, 4, 6\right\}$.
	Let \(\cot \varphi_1 \ge \cot \varphi_2 \ge \dots \ge \cot \varphi_{n-1}\) be the principal curvatures of \(M\), where \(\varphi_1, \dots, \varphi_{n-1} \in (0,\pi)\).

Let \(\gamma(s) = (x(s), y(s), z(s)) = (\sin r(s) \cos \varphi(s), \sin r(s) \sin \varphi(s), \cos r(s))\), \(s \in I\), be a smooth curve parametrized by arc length in the two-dimensional unit sphere, where
\[(r(s), \varphi(s)) \in (0,\pi) \times (\varphi_1 - \frac{\pi}{g}, \varphi_1).\]
We define an immersion \(F\) from \(M \times I\) into \(\mathbb{S}^{n+1}\) as follows:
\begin{equation}\label{DefinitionOfF}
\begin{aligned}
F: M \times I \quad &\to \quad \mathbb{S}^{n+1}, \\
(p, s) \quad &\mapsto \quad (p\,x(s) + N(p)\,y(s), \; z(s)).
\end{aligned}
\end{equation}
One can verify that
\[
\nu(p, s) = ((yz' - y'z)\,p + (x'z - xz')\,N, \; xy' - x'y)
\]
is a unit normal vector field to \(F\), and that the principal curvatures of \(F\) are given by
\begin{equation}\label{eq:2025-12-02-11}
\begin{aligned}
&\kappa_i = \csc r \, \cot(\varphi_i - \varphi) \cos \alpha + \cot r \, \sin \alpha, \quad i = 1, \cdots, n-1,\\
&\kappa_n = \alpha' + \cot r \, \sin \alpha,
\end{aligned}
\end{equation}
where \(\alpha(s)\) denotes the angle between the tangent vector of the curve \(\gamma(s)\) and \(\frac{\partial}{\partial r}\).
For \(g=2\) and the case where the two multiplicities of principal curvature are equal, the computation can be found in \cite{HW}.
\begin{remark}
When \(g = 1\), the hypersurface constructed here is precisely a rotation hypersurface in the sphere. Indeed, taking \(M = \mathbb{S}^{n-1} := \mathbb{S}^n \cap \{ x_{n+1} = 0 \}\) and \(N = (0, \cdots, 0, 1)\), a point \(p \in \mathbb{S}^{n-1}\) can be written as \((\bar p, 0)\) with \(|\bar p| = 1\). In this case,
\((p\,x(s) + N(p)\,y(s), \; z(s)) = (\bar p\, x(s), y(s), z(s)).\)
This is exactly the standard immersion form for a rotation hypersurface in the unit sphere; see, e.g., \cite{O, DD, P1}.
\end{remark}
\begin{remark}
Similarly, one can define such immersions into hyperbolic space:
\begin{align*}
F_{-1}: M \times I \quad &\to \quad \mathbb{H}^{n+1} \\
(p, s) \quad &\mapsto \quad (p \sinh r(s) \cos \varphi(s) + N(p) \sinh r(s) \sin \varphi(s), \; \cosh r(s)).
\end{align*}
Here, hyperbolic space is viewed as the Lorentz model, i.e.,
\[
\mathbb{H}^n = \left\{ x \in \mathbb{R}^{n+2}_1 \mid \langle x, x \rangle_1 := x_1^2 + \cdots + x_{n+1}^2 - x_{n+2}^2 = -1 \right\}.
\]
Denote \(x(s) = \sinh r \cos \varphi\), \(y(s) = \sinh r \sin \varphi\), \(z(s) = \cosh r\).
If \(\gamma(s) = (r(s), \varphi(s))\) is a curve parametrized by arc length with respect to the hyperbolic plane metric (i.e., \(\dd r^2 + \sinh^2 r \, \dd\varphi^2\)), then a unit normal vector field \(\nu(p, s)\) to this immersion is given by:
\[
\nu(p, s) = \left( (z y' - z' y) \, p + (z' x - z x') \, N, \; x y' - x' y \right),
\]
Using the angle \(\alpha\) (defined as the angle between the tangent vector of the curve \(\gamma(s)\) and the radial vector field \(\frac{\partial}{\partial r}\)), the principal curvatures of \(F_{-1}\) can be expressed as:
\begin{align*}
&\kappa_i = \operatorname{csch} r \, \cot(\varphi_i - \varphi) \cos \alpha + \coth r \, \sin \alpha, \quad i = 1, \cdots, n-1,\\
&\kappa_n = \alpha' + \coth r \, \sin \alpha,
\end{align*}
\end{remark}

\begin{remark}
More generally, one can define such immersions into warped products with the unit sphere:
\begin{align*}
F_{h}: M \times I \quad &\to \quad \tilde{I} \times_h \mathbb{S}^n \\
(p, s) \quad &\mapsto \quad (r(s), \; p \cos \varphi(s) + N(p) \sin \varphi(s)).
\end{align*}
The principal curvatures of this immersion \(F_h\) are given by:
\begin{equation*}
\begin{aligned}
&\kappa_i = \frac{1}{h(r(s))} \cot(\varphi_i - \varphi(s)) \cos \alpha + \frac{h'(r(s))}{h(r(s))} \sin \alpha, \quad i = 1, \cdots, n-1,\\
&\kappa_n = \alpha' + \frac{h'(r(s))}{h(r(s))} \sin \alpha,
\end{aligned}
\end{equation*}
where \(\alpha(s)\) is the angle between the tangent vector of the arc length parametrized curve \(\gamma(s) = (r(s), \varphi(s))\) and \(\frac{\partial}{\partial r}\) with respect to the metric \(\dd r^2 + h(r)^2 \dd\varphi^2\).
\end{remark}

\begin{remark}
Using the Clifford foliation of the unit sphere \(\mathbb{S}^{k+l+1}\), one can also define a class of immersions as follows:
\begin{align*}
F_{\text{c}}: I \times {S}^k \times M \quad &\to \quad \mathbb{S}^{k+l+1} \\
(s, p, q) \quad &\mapsto \quad \left( p \cos r(s), \; (q \cos \varphi(s) + N(q) \sin \varphi(s)) \sin r(s) \right),
\end{align*}
where \(M\) is an isoparametric hypersurface in \(\mathbb{S}^l\).
If \(\gamma(s) = (r(s), \varphi(s))\) is a curve parametrized by arc length with respect to the metric \(\dd r^2 + \sin r^2 \dd\varphi^2\), and \(\alpha(s)\) denotes the angle between the tangent vector of this curve and \(\frac{\partial}{\partial r}\), then a unit normal vector field \(\nu\) to this immersion can be taken as
\[
\nu(s, p, q) = \left( p \sin r \sin \alpha, \; (-\Psi \cos r \sin \alpha + A \cos \alpha) \right),
\]
where \(\Psi = q \cos \varphi + N \sin \varphi\) and \(A = -q \sin \varphi + N \cos \varphi\).
The principal curvatures of the hypersurface \(F_{\text{c}}\) fall into three families:
\begin{itemize}
    \item Principal curvatures corresponding to \(\mathbb{S}^k\) (multiplicity \(k\)):
    \[
    \mu = -\tan r \, \sin \alpha.
    \]
    \item Principal curvatures corresponding to \(M\):
    \[
    \lambda_i = \csc r \, \cot(\varphi_i - \varphi) \cos \alpha + \cot r \, \sin \alpha, \quad i = 1, \dots, l-1.
    \]
    \item Principal curvature corresponding to the profile curve direction (multiplicity \(1\)):
    \[
    \kappa = \alpha' + \cot r \, \sin \alpha.
    \]
\end{itemize}
\end{remark}
It is well known that all distinct principal curvatures of the isoparametric hypersurface \(M\) take the form
\[
\cot \varphi_1,\;\cot\!\left(\varphi_1+\frac{\pi}{g}\right),\;\dots,\;\cot\!\left(\varphi_1+\frac{(g-1)\pi}{g}\right),
\]
and their multiplicities alternate in this ordering. Furthermore, when \(g\) is odd, all multiplicities are equal.
Let \(m_2\) denote the multiplicity of the largest principal curvature of \(M\), and \(m_1\) the multiplicity of the second largest (if \(g=1\), we set \(m_1=m_2\)).
Recalling the the principal curvatures of \(F\):
\begin{equation}\tag{2}
\begin{aligned}
&\kappa_i = \csc r \, \cot(\varphi_i - \varphi) \cos \alpha + \cot r \, \sin \alpha, \quad i = 1, \cdots, n-1,\\
&\kappa_n = \alpha' + \cot r \, \sin \alpha.
\end{aligned}
\end{equation}
Using the summation form of the multiple-angle cotangent formula, one finds that the mean curvature of \(F\) is given by
\begin{align*}
	H&=\alpha'+n \cot r \sin \alpha+\frac{g}{2}\left[ m_2 \cot \left( \frac{g}{2}(\varphi_1-\varphi) \right)-m_1 \tan \left( \frac{g}{2}(\varphi_1-\varphi) \right) \right]\csc r \cos \alpha \\
	&=\alpha'+n \cot r \sin \alpha+\frac{g}{2}\left[ m_2 \tan \left( \frac{g}{2} \theta\right)-m_1 \cot \left( \frac{g}{2}\theta \right) \right]\csc r \cos \alpha,
\end{align*}
where \(\theta=\frac{\pi}{g}-\varphi_1+\varphi\).
Combining the fact that \(\gamma(s)\) is parametrized by arc-length, we know that \(F\) has constant mean curvature \(H\) if and only if \(r,\,\theta,\,\alpha\) satisfy the following ODE
\begin{equation}\label{Initial}
	\begin{aligned}
		r'&= \cos \alpha,\qquad\theta'= \frac{\sin \alpha}{\sin r},\\
		\alpha' &= - n \cot r \sin \alpha+ \frac{g }{2}\left[ m_1 \cot \left( \frac{g }{2}\theta\right)-m_2 \tan \left( \frac{g }{2} \theta\right)  \right]\csc r \cos \alpha +H.
	\end{aligned}
\end{equation}
\begin{remark}
When \(H=0\), the above system of ordinary differential equations is equivalent to the system consisting of \((r')^2+(\theta')^2 \sin^2 r=1\) and
\begin{equation}\nonumber
(r'\theta''-r''\theta')\sin r+(r')^2 \theta'\cos r+n\theta'\cos r-\frac{g}{2}\left(m_1\cot\left(\frac{g}{2}\theta\right)-m_2\tan\left(\frac{g}{2}\theta\right)\right)r'\csc r=0.
\end{equation}
Direct computation shows that a curve \((r(s),\theta(s))\) satisfies this system if and only if its trajectory is a geodesic trajectory under the following conformal metric (see \cite{CarS}):
\[
\sin^{2n-2} r\,\sin^{2m_1}\!\left( \frac{g }{2}\theta \right)\cos^{2m_2}\!\left( \frac{g }{2}\theta \right)\bigl(\sin^2r\, \dd\theta^2+\dd r^2\bigr).
\]
Thus, proving theorem~\ref{MainTheorem} essentially reduces to establishing the existence of a simple closed geodesic under this metric. The present paper employs the shooting method to prove its existence; another possible approach might be the (modified) curve shortening flow, see \cite{DN}.

It is particularly noteworthy that Hsiang \cite{Hs3}, in his study of the Bernstein problem on the sphere, searched for a {\bf special class of  solution curves}: these connect the two singular boundaries of the quotient space (namely \(\left\{ \theta=0\right\}\) and \(\left\{ \theta={\pi}/{g}\right\}\)) and satisfy {\bf orthogonality conditions} at the boundaries. This strict boundary behavior ensures that, when lifted back to the original space, the curve generates a smooth closed hypersurface topologically equivalent to a sphere. Hsiang's breakthrough lay in discovering that, besides the obvious equatorial solution, there exist non?trivial solution curves satisfying the above boundary conditions, thereby constructing embedded minimal spheres that are not totally geodesic, see also \cite{Hs4, Hs5,Tomter1987}.
\end{remark}

	After a reparametrization, one sees that equation~\eqref{Initial} is equivalent to
	\begin{equation*}
		\begin{aligned}
			r' &= \sin r \cos \alpha,\qquad			\theta' = \sin \alpha,\\
			\alpha' &= -n \cos r \sin \alpha + \frac{g }{2}\left[ m_1 \cot \left( \frac{g }{2}\theta\right)-m_2 \tan \left( \frac{g }{2} \theta\right)  \right] \cos \alpha + H \sin r.
		\end{aligned}
	\end{equation*}
	Let \(\xi = \frac{g}{2}\ln \tan \frac{r}{2}\) and \(\vartheta= \frac{g }{2}\theta\), then \(\cos r= -\tanh \frac{2}{g}\xi,\, \sin r = \text{sech}\,\frac{2}{g}\xi\) and \(\xi\), \(\vartheta\), \(\alpha\) satisfy the following ODE
	\begin{align*}
			\xi' &= \frac{g }{2} \cos \alpha,\qquad
			\vartheta' = \frac{g }{2} \sin \alpha,\\
			\alpha' &= \frac{g }{2}\left[ \frac{2n}{g}\tanh \frac{2}{g}\xi \sin \alpha + \left( m_1 \cot \vartheta-m_2 \tan \vartheta \right)\cos \alpha + \frac{2}{g} H\,\text{sech}\,\frac{2}{g}\xi \right],
	\end{align*}
	which is equivalent to
	\begin{equation}\label{SystemOfCMC}
		\left\{ \begin{aligned}
			\xi' &= \sin 2 \vartheta \cos \alpha,\\
			\vartheta' &= \sin 2 \vartheta \sin \alpha,\\
			\alpha' &= m \sin 2 \vartheta \tanh \frac{2}{g}\xi \sin \alpha + 2(m_1 \cos ^2 \vartheta - m_2 \sin^2 \vartheta)\cos \alpha +\frac{2}{g}H \sin 2 \vartheta\,\text{sech}\,\frac{2}{g}\xi,
		\end{aligned}\right.
	\end{equation}
	where \(m:= \frac{2}{g}n\).
	Let \(D=\mathbb{R}\times (0,\frac{\pi}{2})\times \mathbb{R}\).
	In the present paper, we will focus solely on solutions of equation~\eqref{SystemOfCMC} that lie in \(D\).

	When \(H=0\) and \((\xi(t),\vartheta(t))\) is of the form \((\xi,\vartheta(\xi))\), then \(\vartheta(\xi)\) satisfies:
	\begin{equation}\label{eq:2025-07-26-1}
		\frac{\dd^2 \vartheta}{\dd \xi^2}=\left[ 1+\left( \frac{\dd \vartheta}{\dd \xi} \right)^2 \right]\left( m \frac{\dd \vartheta}{\dd \xi}\tanh \frac{2}{g}\xi +  m_1 \cot \vartheta - m_2 \tan \vartheta  \right).
	\end{equation}
	When \(H=0\) and \((\xi(t),\vartheta(t))\) is of the form \((\xi(\vartheta),\vartheta)\), then \(\xi(\vartheta)\) satisfies:
	\begin{equation}\label{Xi}
		\frac{\dd^2 \xi}{\dd \vartheta^2} = -\left[ 1+ \left( \frac{\dd \xi}{\dd \vartheta} \right)^2 \right] \left[ m \tanh \frac{2}{g}\xi +\left( m_1 \cot \vartheta - m_2 \tan \vartheta \right) \frac{\dd \xi}{\dd \vartheta} \right].
	\end{equation}
	\begin{remark}\noindent
		\begin{itemize}
			\item If \(m_1 \cos^2 \vartheta-m_2 \sin^2 \vartheta=0\) then \(\vartheta= \arctan\sqrt{\frac{m_1}{m_2}}=:\vartheta^*\);
			\item The equation \eqref{eq:2025-07-26-1} has a unique constant solution \(\vartheta \equiv \vartheta^*\);
			\item \(m_1 \csc^2 \vartheta^*+ m_2 \sec^2 \vartheta^*=2(m_1+m_2)\);
		\end{itemize}
	\end{remark}
We conclude this section by two elementary propositions.
\begin{proposition}\label{ElemataryPropertyOfSystem}
For solutions of \eqref{SystemOfCMC}, we have:
\begin{enumerate}[{\rm (i)}]
    \item For any $(\xi_0,\vartheta_0,\alpha_0) \in \mathbb{R}^3$, there exists a unique solution $\gamma$ of \eqref{SystemOfCMC} satisfying the initial condition $\gamma(0)=(\xi_0,\vartheta_0,\alpha_0)$. This solution is smooth and its maximal interval of existence is the whole real line, i.e., the solution exists for all time.

    \item Suppose $(\xi_s,\vartheta_s,\alpha_s) $ converges to a point $(\xi_\infty,\vartheta_\infty,\alpha_\infty) $ in $\mathbb{R}^3$. Denote by $\gamma_s$ the solution of \eqref{SystemOfCMC} with initial data $(\xi_s,\vartheta_s,\alpha_s)$, and by $\gamma_\infty$ the solution with initial data $(\xi_\infty,\vartheta_\infty,\alpha_\infty)$. Then $\gamma_s$ converges to $\gamma_\infty$ uniformly on every compact subset of $\mathbb{R}$.

    \item A solution of \eqref{SystemOfCMC} whose initial data lie in $D$ remains in $D$ for all time.
\end{enumerate}
\end{proposition}
\begin{proof}
	The proof is standard so we omit it.
\end{proof}
\begin{proposition}\label{SymmetryOfSystemOfCMC}
Let \(\gamma(t)=(\xi(t),\vartheta(t),\alpha(t))\) be a solution of \eqref{SystemOfCMC} in \(D\). Then
\begin{gather*}
    \hat\gamma(t)=(-\xi(-t),\;\vartheta(-t),\;-\alpha(-t)),\\[4pt]
    \tilde\gamma(t)=(\xi(-t),\;\frac{\pi}{2}-\vartheta(-t),\;\pi-\alpha(-t)) \qquad (\text{when } m_1=m_2),\\[4pt]
    \check{\gamma}(t)=(\xi(t),\;\vartheta(t),\;\alpha(t)+2k\pi), \quad \forall\, k \in \mathbb{Z}
\end{gather*}
are also solutions of \eqref{SystemOfCMC} in \(D\).
\end{proposition}
\begin{proof}
	The proof is a straightforward computation.
\end{proof}

	\section{Existence of periodic curves}
	Throughout this paper \(\arctan(\sqrt{m_1/m_2})\) will be denoted by \(\vartheta^*\).
	To prove theorem~\ref{MainTheorem}, it suffices to find solutions of \eqref{SystemOfCMC} (when \(H=0\)) such that the components \((\xi(t),\vartheta(t))\) form smooth simple closed curves.
	Let us begin by a definition.
	\begin{definition}\label{def:10-25-4}
	For $\delta \in (0,\frac{\pi}{2})$ let $(\xi_{\delta}(t),\vartheta_{\delta}(t),\alpha_{\delta}(t))$ denote the solution of \eqref{SystemOfCMC} with initial condition $(\xi_\delta(0),\vartheta_\delta(0),\alpha_\delta(0))=(0,\delta,0)$. Then:
	\begin{enumerate}[{\rm (i)}]
		\item $\delta$ is said to be type 1 if there is a $T>0$ so that $\xi_{\delta}(T)=0$ and $\vartheta_{\delta}'(t)\ne 0$ for all $t\in (0,T)$. See figure~\ref{TypesOftheProfileCurves}(a) for the schematic.
		\item $\delta$ is said to be type 2 if there is a $T>0$ so that $\vartheta_{\delta}'(T)= 0$ and $\xi_{\delta}(t)\ne0$ for all $t\in (0,T)$. See figure~\ref{TypesOftheProfileCurves}(b), (c) for the schematic.
		\item $\delta$ is said to be type 3 if $\vartheta_{\delta}'(t) \ne 0$ and $\xi_{\delta}(t)\ne0$ for all $t>0$. See figure~\ref{TypesOftheProfileCurves}(d) for the schematic.
	\end{enumerate}
\end{definition}
\noindent
\begin{figure}[htbp]
\centering

\newcommand{\uniformheight}{6cm}

\parbox{0.22\textwidth}{
  \centering
  \includegraphics[height=\uniformheight, width=\textwidth, keepaspectratio]{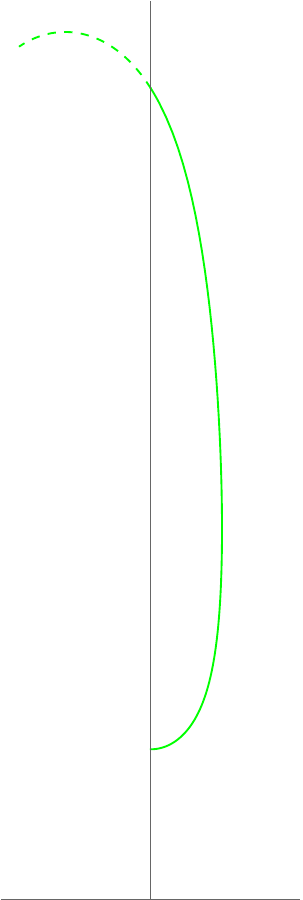}
  \\ \footnotesize{(a)\,Type~1}
}
\hfill
\parbox{0.22\textwidth}{
  \centering
  \includegraphics[height=\uniformheight, width=\textwidth, keepaspectratio]{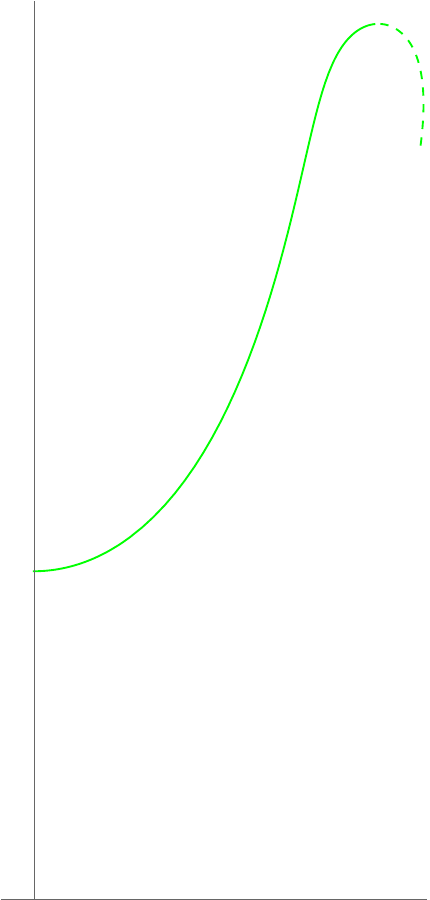}
  \\ \footnotesize{(b)\,Type~2}
}
\hfill
\parbox{0.22\textwidth}{
  \centering
  \includegraphics[height=\uniformheight, width=\textwidth, keepaspectratio]{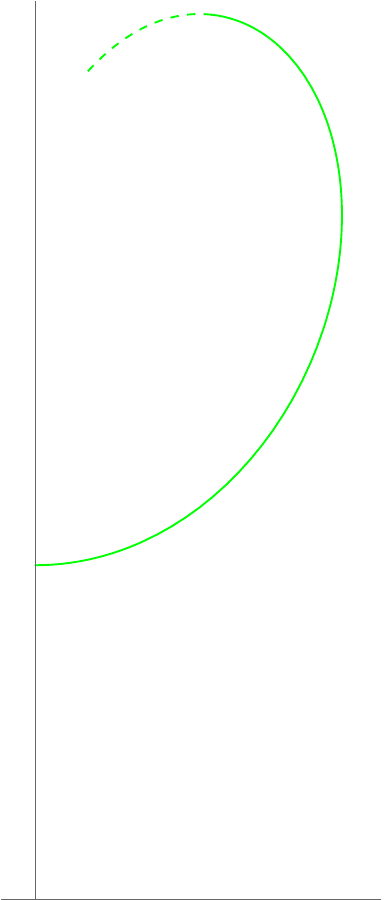}
  \\ \footnotesize{(c)\,Type~2}
}
\hfill
\parbox{0.22\textwidth}{
  \centering
  \includegraphics[height=\uniformheight, width=\textwidth, keepaspectratio]{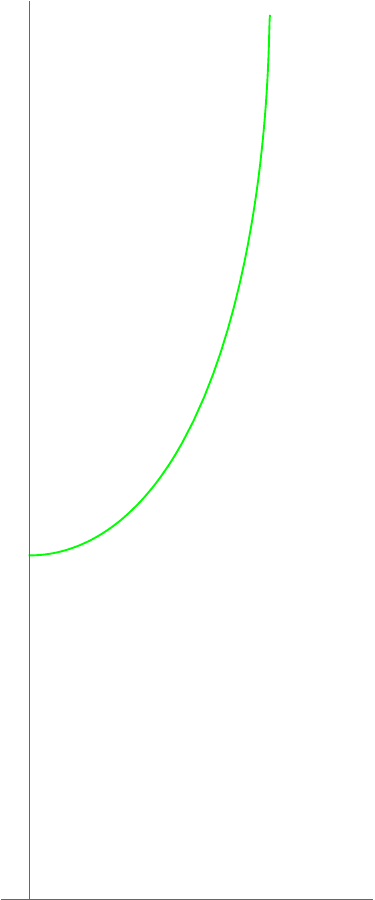}
  \\ \footnotesize{(d)\,Type~3}
}
\vskip2mm
\caption{Types of the profile curves.}
\label{TypesOftheProfileCurves}
\end{figure}
\vskip1mm
Note that these three situations cover all possibilities.
We next define
\[\delta^*:=\sup \left\{ \delta\in(0,\frac{\pi}{2})\mid \delta_0\text{ is type 1 for all }\delta_0\in(0,\delta)\right\}\]
with the convention that \(\delta^* = 0\) if the set is empty.
\begin{proposition}\label{prop:2025-10-03-1}
	Let \(H=0\), then we have:
	\begin{enumerate}[{\rm (i)}]
		\item $\delta^*>0$.
		\item $\delta^*<\vartheta^*$.
		\item $\delta^*$ is not type 3.
	\end{enumerate}
\end{proposition}
This proposition will be proved in section~\ref{sec5},\ref{sec6},\ref{sec7}. To proceed, we state the following lemma, whose proof is straightforward and is therefore omitted.
\begin{lemma}\label{lem:2025-10-03-1}
	Let \(H=0\), then for \((\xi,\vartheta,\alpha)(t)\in D\) one has:
	\begin{enumerate}[{\rm (i)}]
		\item If \(\xi'(t)=0\), then \(\xi''(t)>0\) iff \(\xi(t)<0\); \(\xi''(t)<0\) iff \(\xi(t)>0\); \(\xi''(t)=0\) iff \(\xi\) is identically zero.
		\item If \(\vartheta'(t)=0\), then \(\vartheta''(t)>0\) iff \(\vartheta(t)<\vartheta^*\); \(\vartheta''(t)<0\) iff \(\vartheta(t)>\vartheta^*\); \(\vartheta''(t)=0\) iff \(\vartheta\) is identically \(\vartheta^*\).
	\end{enumerate}
\end{lemma}
This now gives:
\begin{proposition}\label{prop:2025-10-03-2}
	Let \(H=0\) then \(\delta^*\) is type 1 and type 2.
\end{proposition}
\begin{proof}
	Using the  continuous dependence of the solutions to \eqref{eq:2025-10-02-1} on the initial conditions, the proof is standard, see also \cite{R}.
\end{proof}
Applying proposition~\ref{SymmetryOfSystemOfCMC}, this gives the existence of periodic curves, and then proves theorem~\ref{MainTheorem}.
The following figure shows a periodic profile curve for \(g=4,\,m_1=4,\,m_2=5\) and \(\delta^*\approx 0.341639\).
\begin{figure}[htbp]
	\centering
	\includegraphics[scale=0.38]{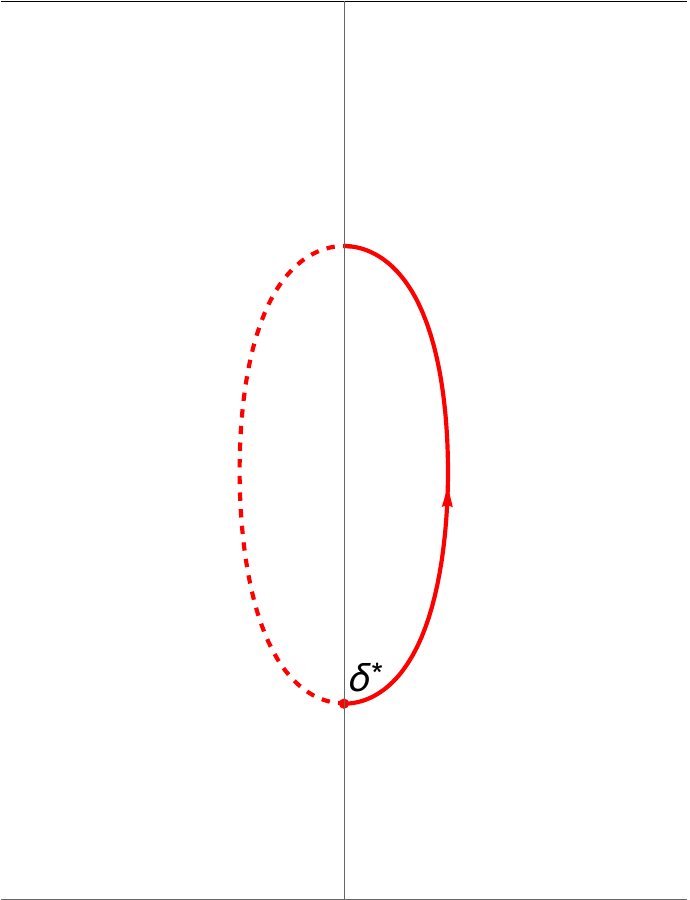}
	\caption{The existence of the periodic profile curve}
	\label{PeriodicCurve}
\end{figure}

	\section{Preparations for proposition~\ref{prop:2025-10-03-1}}
	Recall that the equation \eqref{Xi} is
	\[\frac{\dd^2 \xi}{\dd \vartheta^2} = -\left[ 1+ \left( \frac{\dd \xi}{\dd \vartheta} \right)^2 \right] \left[ m \tanh \frac{2}{g}\xi +\left( m_1 \cot \vartheta - m_2 \tan \vartheta \right) \frac{\dd \xi}{\dd \vartheta} \right].\]

	\begin{lemma}\label{lem:2025-09-17-1}
		Let \(\xi(\vartheta),\, a \leq \vartheta \le \vartheta^*\) be a solution of~\eqref{Xi} with \(\xi'(a)=1\), \(\xi(\vartheta^*)>0\) and \(\xi'(\vartheta) \ge 0\) for \(\vartheta\in [a,\vartheta^*]\) (\(\xi'(a)=-1\), \(\xi(\vartheta^*)<0\) and \(\xi'(\vartheta) \le 0\) for \(\vartheta\in [a,\vartheta^*]\)).
		If \(|\xi(\vartheta^*)|\) and \(a\) are sufficiently small, then \(\xi(a) \le 0\) ($\xi(a) \ge 0$).
	\end{lemma}
	\begin{proof}
		Suppose \(\xi'(a)=1\), \(\xi(\vartheta^*)>0\) and \(\xi'(\vartheta) \ge 0\) for \(\vartheta\in [a,\vartheta^*]\).
		We assume \(\xi(a)>0\) for contradiction.
		The first step is to show that for sufficiently small \(|\xi(\vartheta^*)|\) and \(a\), the uniform estimates
		\begin{equation}\label{eq:2025-09-17-4}
			c_1\, \xi(\vartheta^*) <a ^{m_1}<c_2\, \xi(\vartheta^*).
		\end{equation} holds for some constant \(c_1\), \(c_2\).
		Specifically, we may take \[c_1=\frac{m \vartheta^*}{8g}\mathrm{e}^{-17} \sin^{m_1}\frac{\vartheta^*}{2}\cos^{m_2}\vartheta^*\qquad \text{and}\qquad c_2=\frac{1}{\vartheta^*}2^{m_1+2}\mathrm{e}^{18}\sin^{m_1}\vartheta^*\cos^{m_2}\frac{\vartheta^*}{2}.\]
		Since \(\xi'(\vartheta)\ge 0\) for all \(\vartheta\in(a,\vartheta^*)\) we obtain that \(\xi(\vartheta)>0\) for such \(\vartheta\).
		Then from \eqref{Xi} one gets that \(\xi''(\vartheta)<0\) for all \(\vartheta\in(a,\vartheta^*)\) and then \( \xi'(\vartheta)\) is strictly decreasing in this interval.
		Monotonicity of \(\xi'(\vartheta)\) and mean value theorem then imply
		\begin{equation}\label{eq:2025-09-18-1}
			\xi' \left( \frac{a+\vartheta^*}{2} \right)<\xi(\vartheta^*)\frac{2}{\vartheta^*-a},\qquad \xi \left( \frac{a+\vartheta^*}{2}\right)>\frac{1}{2}\xi(\vartheta^*) .
		\end{equation}
		Combining the second inequality with the graph ODE \eqref{Xi} yields (recall that \(\xi(\vartheta^*)\) are sufficiently small) additionally the lower bound for \(\xi'\) at \((a+\vartheta^*)/2\). So there are uniform constant \(\tilde{c}_1,\,\tilde{c}_2>0\) so that for sufficiently small \(\xi(\vartheta^*)\) and \(a\):
		\begin{equation}\label{eq:2025-09-17-2}
			\tilde{c}_1\, \xi(\vartheta^*) < \xi'\left( \frac{a+\vartheta^*}{2} \right)< \tilde{c}_2\,\xi(\vartheta^*).
		\end{equation}
		One may take \(\tilde{c}_1=\frac{m\vartheta^*}{8g}\) and \(\tilde{c}_2=\frac{4}{\vartheta^*}\).
		Rewriting the graph ODE as
		\begin{equation}\label{eq:2025-09-17-3}
			\frac{\xi''}{\xi'(1+(\xi')^2)}=- \frac{m \tanh \frac{2}{g}\xi}{\xi'}-G(\vartheta),
		\end{equation}
		where \(G(\vartheta):= m_1 \cot \vartheta-m_2 \tan \vartheta\).
		One notes that the first term on the right-hand side is \(O(1)\) in the interval \((a,\frac{a+\vartheta^*}{2})\) by \eqref{eq:2025-09-17-2} and monotonicity of \(\xi'\), and further that the left-hand side has \(\ln(\frac{\xi'}{\sqrt{1+(\xi')^2}})\) as an anti-derivative.
		Integrating \eqref{eq:2025-09-17-3} over \((a,\frac{a+\vartheta^*}{2})\) then gives
		\[\ln \left( \xi' \left( \frac{a+\vartheta^*}{2} \right)\right)= m_1 \ln \sin a+O(1) .\]
		The approximation \(\sin x = x+ O(x^3)\) then gives the bounds \eqref{eq:2025-09-17-4} from \eqref{eq:2025-09-17-2}.

		For the next step, note that \(G(\vartheta)-\xi'(\vartheta)^{-1}\) becomes negative as \(\vartheta\to \vartheta^*\) and \(G(a)-1\) as \(\vartheta\to a\), which for \(a\) small enough will be positive.
		Hence for small enough \(a\), there is a \(b\in(a,\vartheta)\) so that \(G(b)=\xi'(b)^{-1}\).
		From \(\eqref{eq:2025-09-18-1}\) and monotonicity of \(\xi'\) one knows that \(\xi'(\vartheta^*)\) converges to zero as \(\xi(\vartheta^*)\) and \(a\) approach zero.
		Consequently, \(\xi(\vartheta)\) and its derivatives converge to zero uniformly on compacta, which then follows that \(b\) converges to zero.
		
		We let \[q=\frac{b}{a}\] and show next that there is a uniform constant \(c_3>0\) so that \(q^{m_1} \ge {c_3}/{b}\), and hence that \(q\) grows unboundedly as \(\xi(\vartheta^*)\) and \(a\) approach zero.
		One may take \(c_3=m_1\mathrm{e}^{-18}2^{-(m_1+1)}\).

		To show this we again integrate \eqref{eq:2025-09-17-3}, this time from \(\vartheta \ge a\) to \(b\).
		The result is:
		\[\ln \xi'(b)-\ln \xi'(\vartheta) = m_1( \ln \sin \vartheta - \ln \sin b)+O(1).\]
		Noting that \(\xi'(b)=G(b)^{-1}= \frac{1}{m_1} b+O(b^3)\). Using the expansion of sine close to \(0\) implies the existence of \(c_3,\,c_4>0\) so that for small enough \(\xi(\vartheta^*)\) and \(a\):
		\begin{equation}\label{eq:2025-09-18-2}
			c_4\,\xi'(\vartheta) \ge \left( \frac{b}{\vartheta} \right)^{m_1} b \ge c_3\, \xi'(\vartheta).
		\end{equation}
		One may take \(c_4=m_1\mathrm{e}^{18}2^{m_1+1}\).
		Taking \(\vartheta=a\) in~\eqref{eq:2025-09-18-2} shows \(q^{m_1} \ge c_3/b\).
		On the other hand, if one integrate \(\xi'(\vartheta)\) over \((a,b)\) one finds by \eqref{eq:2025-09-18-2} that
		\[c_4(\xi(b)-\xi(a)) \ge \begin{cases}
			b^2 \ln q & m_1=1\\
			\frac{1}{m_1-1}b^2(q^{m_1-1}-1)& m_2>1.
		\end{cases}\]
		Noting that \(b^2=\frac{a^{m_1}}{b^{m_1}}q^{m_1}b^2>c_1c_3b^{1-m_1}\xi(\vartheta^*)\) and using unboundedness of \(q\) one sees that \(\xi(b)-\xi(a)\) will be larger than \(\xi(\vartheta^*)\) for sufficiently small \(\xi(\vartheta^*)\) and \(a\), contradicting our assumption that \(\xi(a) \ge 0\).
		The proof of the other case is similar.

	\end{proof}
	The following lemma can be proved similarly.
	\begin{lemma}\label{lem:2025-10-02-3}
		Let \(\xi(\vartheta),\, \vartheta^* \le \vartheta \le a \) be a sequence of solution of \eqref{Xi} with \(\xi'(a)=-1\), \(\xi(\vartheta^*)>0\) and \(\xi'(\vartheta) \le 0\) for \(\vartheta\in [\vartheta^*,a]\) (\(\xi'(a)=1\), \(\xi(\vartheta^*)<0\) and \(\xi'(\vartheta) \ge 0\) for \(\vartheta\in [\vartheta^*,a]\)).
		If \(|\xi(\vartheta)|\) and \(\frac{\pi}{2}-a\) are sufficiently small, then \(\xi(a) \le 0\) ($\xi(a) \ge 0$).
	\end{lemma}
	Next we recall the system \eqref{SystemOfCMC} for \(H=0\):
	\begin{equation}\label{eq:2025-10-02-1}
		\left\{ \begin{aligned}
			\xi' &= \sin 2 \vartheta \cos \alpha,\\
			\vartheta' &= \sin 2 \vartheta \sin \alpha,\\
			\alpha' &= m \sin 2 \vartheta \tanh \frac{2}{g}\xi \sin \alpha + 2(m_1 \cos ^2 \vartheta - m_2 \sin^2 \vartheta)\cos \alpha,
		\end{aligned}\right.
	\end{equation}
	Note that \(\xi'(t),\,\vartheta'(t),\,\alpha'(t)\) is bounded, hence all higher derivatives of \(\xi,\,\vartheta,\,\alpha\) are also bounded.
	\begin{lemma}\label{lem:2025-10-02-1}
		Let \(H=0\) and \(t_0\in \mathbb{R}\), then $\vartheta(t_0)<\vartheta^*$ ($\vartheta(t_0)>\vartheta^*$) and \(\vartheta'(t_0)>0\)(\(\vartheta'(t_0)<0\)) imply the existence of a time $T\in(0,\infty)$ such that $\vartheta(t_0+T)=\vartheta^*$.
	\end{lemma}
	\begin{proof}
		Suppose \(\vartheta(t_0)<\vartheta^*\), \(\vartheta'(t_0)>0\) and \(\vartheta(t)<\vartheta^*\) for all \(t>t_0\).
		We shall derive a contradiction.
		Let us first note that $\vartheta'(t) > 0$ for all $t > t_0$ by lemma~\ref{lem:2025-10-03-1}.
		If there is a sequence \(\left\{ t_i\right\}_{i=1}^\infty\) increasing to \(\infty\) such that \(\xi'(t_i)=0\), then from the increasing monotonicity of \(\vartheta(t)\) one sees that  \(\vartheta'(t_i)= \sin 2 \vartheta(t_i)>0\) is bounded away from zero for all \(i\).
		Therefore the boundedness of \(\vartheta''(t)\) leads directly to a contradiction.
		If \(\xi'(t)\ne 0\) for large \(t\), we then assume \(\xi'(t)>0\) and \(\alpha(t)\in(0,\frac{\pi}{2})\) for large \(t\) without loss of generality. It follows that \(\lim_{t\to \infty}\xi(t)=\infty\) or \(\lim_{t\to \infty}\xi(t)\) is finite.
		The latter yields that \(\xi'\) and \(\vartheta'\) converge to zero and then \(\vartheta(t)\) converges to \(\frac{\pi}{2}\), which is a contradiction.
		For the former, one sees that \(\alpha'(t)>0\) for large \(t\), which follows that \(\alpha(t)\) is bounded away from \(0\) for large \(t\).
		Hence \(\vartheta'(t)\) is also bounded away from \(0\) for large \(t\),  which again leads to a contradiction.
		The proof for the other case proceeds analogously.
	\end{proof}

	\begin{lemma}\label{lem:2025-10-02-2}
		Let \(H=0\) and \(t_0\in \mathbb{R}\), then $\xi(t_0)<0$ ($\xi(t_0)>0$) and \(\xi'(t_0)>0\)(\(\xi'(t_0)<0\)) imply the existence of a time $T\in(0,\infty)$ so that $\xi(t_0+T)=0$.
	\end{lemma}
	\begin{proof}
		Suppose \(\xi(t_0)<0\),\,\(\xi'(t_0)>0\) and \(\xi(t)<0\) for all \(t >t_0\).
		We shall derive a contradiction.
		Let us first note that $\xi'(t) > 0$ for all $t > t_0$  by lemma~\ref{lem:2025-10-03-1}.
		The boundedness of \(\alpha'(t)\) implies that \(\xi''(t)\) is bounded, and hence \(\lim_{t\to\infty}\xi'(t)=0.\)

		If there is a sequence \(\left\{ t_i\right\}_{i=1}^\infty\) increasing to \(\infty\) such that \(\vartheta'(t_i)=0\), then from \(\lim_{t\to\infty}\xi'(t)=0\) one sees that \[\lim_{i\to\infty}\sin 2 \vartheta(t_i)=0\] necessarily.
		From lemma~\ref{lem:2025-10-03-1} and lemma~\ref{lem:2025-10-02-1}, the existence of such sequence \(\left\{ t_i\right\}_{i=1}^\infty\) implies that existence of the sequence \(\left\{ \tilde t_i\right\}_{i=1}^\infty\) increasing to \(\infty\) such that \(\vartheta(\tilde{t}_i)=\vartheta^*\).
		Note that \(\lim_{i\to\infty}\sin \alpha(\tilde{t}_i)=1\) up to a subsequence, which follows from \(\lim_{t\to}\xi'(t)=0\).
		We shall show that \[\lim_{i\to\infty}\xi(\tilde{t}_i)=0.\]
		Assume \(\lim_{i\to\infty}\xi(\tilde{t}_i)<0\),
		then \(\lim_{i\to\infty}\alpha'(\tilde{t}_i)<2b<0\) for some \(b\).
		The boundedness of \(\alpha'(t)\) implies that \(\alpha''(t)\) is bounded, which gives the existence of \(\varepsilon>0\) such that \(\alpha'(\tilde{t}_i+t)<b<0\) for large \(i\) and \(t\in[0,\varepsilon]\).
		Without loss of generality, we may assume \(\lim_{i\to\infty}\alpha(\tilde{t}_i)=\frac{\pi}{2}\).
		Then \[\alpha(\tilde{t}_i+\varepsilon)<\frac{\pi}{2}+\frac{b \varepsilon}{2}\] for large \(i\).
		By possibly making \(\varepsilon\) smaller, one can ensure that \(\vartheta(\tilde{t}_i+t)\) is bounded away from \(\left\{ 0,\frac{\pi}{2}\right\}\) for all \(i\) and \(t\in[0,\varepsilon]\).
		Combining with \(\lim_{i\to\infty}\sin 2 \vartheta(t_i)=0\) one sees that \(\vartheta'(\tilde{t}_i+t)>0\) for all \(t\in[0,\varepsilon]\) and large \(i\).
		Hence for large \(i\)
		\[0<\alpha(\tilde{t}_i+\varepsilon)<\frac{\pi+b \varepsilon}{2},\]
		which contradicts \(\lim_{t\to\infty}\xi'(t)=0\).

		Now we have proved that \(\lim_{i\to\infty}\xi(\tilde{t}_i)=0\).
		So, without loss of generality, one can find a sequence \(\left\{ \xi_i,\, \vartheta^* \le \vartheta < b_i\right\}\) of solutions of \eqref{Xi} with \(\lim_{\vartheta\to b_i}\xi'_i(\vartheta)=\infty\), \(\lim_{i\to\infty}\xi_i'(\vartheta^*)=0\), \(\xi_i(\vartheta^*)\) increasing to zero and \(\xi_i'(\vartheta) \ge 0\) for \(\vartheta\in[\vartheta^*,b_i)\).
		Note that the solution of \eqref{Xi} with initial condition \(\xi(\vartheta^*)=\xi'(\vartheta^*)=0\) is a constant solution.
		Hence, by the continuous dependence of the solutions to \eqref{eq:2025-10-02-1} on the initial conditions, there exists for large \(i\) a sequence \(a_i\in(\vartheta^*,b_i)\) such that \(\xi_i'(a_i)=1\) and \(\lim_{i\to\infty}a_i=\frac{\pi}{2}\).
		Therefore we arrive at a contradiction by applying lemma~\ref{lem:2025-10-02-3}.

		If \(\vartheta'(t)\ne 0\) for large \(t\), we then assume \(\vartheta'(t)>0\) and \(\alpha(t)\in(0,\frac{\pi}{2})\) for large \(t\) without loss of generality.
		It follows that both \(\xi(t)\) and \(\vartheta(t)\) converge, from which one gets that \(\xi'(t)\) and \(\vartheta'(t)\) converges to zero and then \(\vartheta(t)\) converges to \(\frac{\pi}{2}\).
		Hence from \eqref{eq:2025-10-02-1} one gets \(\alpha'(t)<0\) and then \(\alpha(t)\) is bounded away from \(\frac{\pi}{2}\) for large \(t\).
		Again from \eqref{eq:2025-10-02-1} one obtains that \(\alpha'(t)\) is also bounded away from zero, which is impossible.
		The proof for the other case proceeds analogously.
	\end{proof}
	
	\section{Proof of first item of proposition \ref{prop:2025-10-03-1}}\label{sec5}
	Throughout this section $\xi,\,\vartheta$ and $\alpha$ will denote the components of the solution of \eqref{SystemOfCMC} with initial condition $(\xi,\vartheta,\alpha)(0)=(0,\delta,0)$.
	\begin{lemma}\label{lem:2025-10-01-2}
		Let \(H=0\) and \(i\) be any positive integer.  Then there is a \(\delta_i\in(0,\frac{\pi}{2})\) such that for all \(0<\delta< \delta_i\), one has \(\xi(s)<\frac{2 \vartheta^*}{i}\) on any open interval \(I\) starting at \(0\), as long as \(\vartheta'(s)>0\), \( \vartheta(s) <\vartheta^*\) and \(\alpha'(s)>0\) for \(s\in I\).
	\end{lemma}
	\begin{proof}
		For the proof of the lemma, it is more convenient to work with the following system of ODE:
		\begin{equation}\label{eq:2025-10-01-2}
		\left\{ \begin{aligned}
			\xi' &=  \cos \alpha,\\
			\vartheta' &=  \sin \alpha,\\
			\alpha' &= m  \tanh \frac{2}{g}\xi \sin \alpha + (m_1 \cot \vartheta - m_2 \tan \vartheta)\cos \alpha,
		\end{aligned}\right.
	\end{equation}
	We use a blow-up analysis, so we introduce the following variables:
	\begin{equation}\nonumber
			\Xi(t)=\frac{1}{\delta} \xi(\delta t),\qquad
			\Theta(t)= \frac{1}{\delta}(\vartheta(\delta t)-\delta),\qquad
			\Psi(t) = \alpha(\delta t).
	\end{equation}
	where \(\delta\in(0,\frac{\pi}{2})\).
	From \eqref{eq:2025-10-01-2}, they satisfies:
	\begin{equation}\label{eq:2025-10-01-3}
		\left\{\begin{aligned}
			\Xi'&= \cos \Psi,\\
			\Theta'&= \sin \Psi,\\
			\Psi'&=  \frac{m_1}{\Theta+1}f(\delta(\Theta+1))\cos \Psi+\delta \left[ m \tanh \left( \frac{2 \delta}{g} \Xi \right)\sin \Psi- m_2 \tan \left( \delta(\Theta+1) \right) \cos \Psi \right],
		\end{aligned}\right.
	\end{equation}
	where \(f(x)=x \cot x\).
	Consider this ODE system with initial conditions:
	\[\Xi(0)=0,\quad \Theta(0)= 0,\quad \Psi(0)=0.\]
	For \(\delta=0\), this initial value problem can be solved explicitly, and one gets that \(\Psi(t)= \arcsin \Theta'(t)\), where \(\Theta(t)\) is the inverse function of
	\[t=\int_{0}^{\Theta}\frac{(1+x)^{m_1}}{\sqrt{(1+x)^{2m_1}-1}}\dd x.\]
	An elementary analysis the explicit solution \(\Psi(t)\) shows that \[\Psi'(t)>0 \text{ for } t\in[0,\infty),\qquad \lim_{t\to \infty} \Psi(t)=\frac{\pi}{2}.\]
	Hence one has the following claim from the parameter dependence of \eqref{eq:2025-10-01-3}:
	
	\vskip 2mm
	\noindent
	\textbf{Claim.}\ {\it For any $i \in \mathbb{N}$, there is a \ $T_i > 0$ and a \ $\delta_i >0$ such that for all $0< \delta < \delta_i$,
	one has $\alpha'(s)>0$ for \(s\in [0, T_i \delta]\), and at $s = T_i\,\delta$,
	$\arctan i < \alpha < \pi/2 $,  $\xi = O(\delta)$ and $\vartheta = \delta + O(\delta)$.}
	\vskip2mm

	By possibly making $\delta_i$ smaller, we can ensure that $$\xi(T_i\,\delta) = O(\delta) <\frac{\vartheta^*}{i} \text{ and } \vartheta(T_i\,\delta) = \delta + O(\delta) <1$$ for $0<\delta < \delta_i$.
	Consider the tangent line $l_{i,\delta} : Y - \vartheta(T_i\,\delta)
	= \tan \theta_\delta(T_i\,\delta)(X - \xi(T_i\,\delta))$ of the curve $(\xi(s),\vartheta(s))$ at $s = T_i\,\delta$.
	Let \(Y= \vartheta^*\) in the equation of $l_{i,\delta}$,
	we get	
	\[X= \frac{1}{\tan \alpha(T_i \delta)}(\vartheta^*-O(\delta)-\delta)+O(\delta)<\frac{2 \vartheta^*}{i},\]
	which completes the proof.
	\end{proof}
	Before proceeding, let us note that on an open interval \(I\) starting at \(0\) where \(\vartheta'(t)>0\), \(\vartheta(t)<\vartheta^*\) and \(\xi'(t)>0\), one has \(\alpha'(t)>0\) for \(t\in I\) from \eqref{eq:2025-10-02-1}.

	\begin{lemma}\label{lem:2025-10-02-4}
		Let \(H=0\) and \(\delta>0\) be small enough, there is time \(T_1>0\) so that \(\xi'(T_1)=0\), \(\vartheta(t)<\vartheta^*\) and \(\vartheta'(t)>0\) for \(t\in(0,T_1]\).
	\end{lemma}
	\begin{proof}
		For \(\delta<\vartheta^*\) one has \(\alpha'(0)>0\) and then \(\vartheta'(t)>0\) for small \(t\), which follows that there is a \(T_2>0\) such \(\vartheta(T_2)=\vartheta^*\) by lemma~\ref{lem:2025-10-02-1}.
		Let \(T_2\) be the first such time, then \(\vartheta'(t)>0\) for \(t\in(0,T_2)\) by lemma~\ref{lem:2025-10-03-1}.
		Suppose, for the sake of contradiction, that this lemma is not true.
		Then there is a sequence of initial conditions \(\left\{ \tilde{\delta}_i\right\}\) decreasing to zero such that for the corresponding solutions \((\xi_i, \vartheta_i, \alpha_i)\), we have \(\xi'_i(t)>0\) for \(t\in(0,T_2)\).
		It further implies that \(\alpha'_i(t)>0\) for \(t\in(0,T_2)\) from \eqref{eq:2025-10-02-1}.
		Up to a subsequence, we may assume \(\tilde{\delta}_i<\delta_i\), where \(\delta_i\) are as in lemma~\ref{lem:2025-10-01-2}.
		Then \(\lim_{i\to\infty}\xi_i(T_2)=0\) by lemma~\ref{lem:2025-10-01-2}.
		To proceed let us note that there is a \(T_0\in(0,T_2)\) such that \(\alpha_i(T_0)=\frac{\pi}{4}\) and \(\lim_{i\to\infty}\vartheta_i(T_0)=0\) by the claim in the proof of lemma~\ref{lem:2025-10-01-2}.
		An application of lemma~\ref{lem:2025-09-17-1} now leads to a contradiction, thus completing the proof.
	\end{proof}

	\begin{lemma}\label{lem:2025-10-02-5}
		Let \(H=0\) and \(\delta\) be small enough, then \(\delta\) is type \(1\).
	\end{lemma}
	\begin{proof}
		For small \(\delta>0\), let \(T_1>0\) be the first time defined in lemma~\ref{lem:2025-10-02-4}.
		Hence \(\xi'(t)>0\) for \(t\in(0,T_1)\) and then by lemma~\ref{lem:2025-10-03-1} \(\xi''(T_1)<0\), which follows that there is a \(T_4>T_1\) such that \(\xi(T_4)=0\) by lemma~\ref{lem:2025-10-02-2}.
		Let \(T_4\) be the first such time, then \(\xi'(t)<0\) for \(t\in(T_1,T_4)\) by lemma~\ref{lem:2025-10-03-1}.
		So it suffices to prove that \(\vartheta'(t)\ne 0\) for \(t\in(T_1,T_4)\) and small enough \(\delta>0\).
		Suppose, by contradiction, that this fails.
		Then there is a sequence of initial conditions \(\left\{ \tilde{\delta}_i\right\}\) decreasing to zero such that, for the corresponding solutions \((\xi_i, \vartheta_i, \alpha_i)\)  and some \(T_3\in(T_1,T_4)\) depending on \(i\), \(\vartheta'_i(T_3)=0\) and \(\vartheta'_i(t)>0\) for all \(t\in(T_1,T_3)\).
		Up to a subsequence,  we may assume \(\tilde{\delta}_i<\delta_i\), where \(\delta_i\) are as in lemma~\ref{lem:2025-10-01-2}.
		Then \(\lim_{i\to\infty}\xi_i(T_1)=0\) by the same lemma.
		To proceed let us note that \(\vartheta_i(T_3)>\vartheta^*\) by lemma~\ref{lem:2025-10-03-1}, and that there is a unique \(T_2\in(T_1,T_3)\) depending on \(i\) such that \(\vartheta_i(T_2)=\vartheta^*\) by lemma~\ref{lem:2025-10-02-4}.
		We shall prove that
		\[\lim_{i\to\infty}\xi_i(T_2)=0,\qquad \lim_{i\to\infty}\alpha_i(T_2)=\frac{\pi}{2}.\]
		For the former, it follows from \(\xi_i(t) \le\xi_i(T_1)\) for all \(t\in(0,T_4)\) and \(\lim_{i\to\infty}\xi_i(T_1)=0\).
		For the latter, we use a contradiction argument by assuming, after passing to a subsequence if necessary, that \(\alpha_i(T_2) \in (\frac{\pi}{2}, \pi)\) and is bounded away from \(\frac{\pi}{2}\) for all \(i\).
		Then one has from \eqref{eq:2025-10-02-1} that \(\xi_i'(T_2)<0\) is bounded away from zero.
		Note that \(\xi''(t)\) is bounded, which gives the existence of sufficiently small \(\varepsilon,b>0\) such that \(\xi_i'(t)<-b\) for all \(t\in [T_2-\varepsilon,T_2]\) and large \(i\).
		This, however, is impossible since \(\xi_i(t) \le \xi_i(T_1)\) for all \(t\in [0,T_2]\) and \(\lim_{i\to\infty}\xi_i(T_1)=0\), so $\lim_{i\to\infty}\alpha_i(T_2)=\frac{\pi}{2}$.
		We can now obtain a contradiction by applying lemma~\ref{lem:2025-10-02-3} using the same reasoning as at the end of the penultimate paragraph in the proof of lemma~\ref{lem:2025-10-02-2}.
	\end{proof}
	This lemma proves proposition \ref{prop:2025-10-03-1}(i).

	\section{Proof of second item of proposition \ref{prop:2025-10-03-1}}\label{sec6}
	In what follows $\xi,\,\vartheta$ and $\alpha$ will denote the components of the solution of \eqref{SystemOfCMC} with initial condition $(\xi,\vartheta,\alpha)(0)=(0,\vartheta^*-\varepsilon,0)$.
	Let us first consider the case of \(g=4\).
	We begin by recalling the system \eqref{Initial} for \(H=0\):
	\begin{equation}\label{eq:2025-10-04-1}
		\left\{ \begin{aligned}
			r'&= \cos \alpha,\\
			\theta'&= \csc r\sin \alpha,\\
			\alpha' &= - n \cot r \sin \alpha+ \frac{g }{2}\left( m_1 \cot \left( \frac{g }{2}\theta\right)-m_2 \tan \left( \frac{g }{2} \theta\right)  \right)\csc r \cos \alpha.
		\end{aligned}\right.
	\end{equation}

	If \((r(s),\theta(s))\) is of the form \((r,\theta(r))\) then from \eqref{eq:2025-10-04-1} one gets that \(\theta(r)\) satisfies:
	\begin{equation}\label{eq:2025-10-04-2}\nonumber
		\cos r \frac{\dd \theta}{\dd r}+ \sin r \frac{ \dd^2 \theta}{\dd r^2} = \left[ 1+ \sin^2 r \left( \frac{\dd \theta}{\dd r} \right)^2 \right]\left[ -n \frac{\dd \theta}{\dd r}\cos r + \frac{g }{2}\left( m_1 \cot \left( \frac{g }{2}\theta\right)-m_2 \tan \left( \frac{g }{2} \theta\right)  \right)\csc r \right].
	\end{equation}
	Linearizing this equation at the constant solution \(\theta=\frac{2 \vartheta^*}{g}\) yields the equation
	\begin{equation}\nonumber
		\sin^2 r w''+(n+1)\cos r \sin r w'+g(n-1)w=0,
	\end{equation}
	Here we have used the fact that \[m_1 \csc^2 \vartheta^*+m_2 \sec^2 \vartheta^*=2(m_1+m_2),\qquad n-1=\frac{g}{2}(m_1+m_2).\]
	Next we note that the above linearized equation for \(g=4\) is equivalent to
	\begin{equation}\label{eq:2025-09-15-3}
		\frac{\dd}{\dd r} \left( \sin^{n+1} r  \frac{\dd w}{\dd r} \right) + 4(n-1) \sin^{n-1} r\,  w = 0.
	\end{equation}
	Consider the following equation, distinct from \eqref{eq:2025-09-15-3}:
	\begin{equation}\label{eq:2025-09-15-2}
		\frac{\dd}{\dd r} \left( \sin^{n} r  \frac{\dd w}{\dd r} \right) + 4(n-1) \sin^{n} r\,  w = 0.
	\end{equation}
	One finds that it can be transformed via the substitution \(x = \cos r\)
	into the following Legendre type equation:
	\begin{equation}\label{eq:2025-09-14-1}
		(1-x^2)\frac{\dd^2w}{\dd x^2}-(n+1)x \frac{\dd w }{\dd x}+4(n-1)w=0.
	\end{equation}

	We need a theorem proved by Leighton \cite{Leighton}:
	\begin{theorem}[Leighton]\label{thm:2025-09-14-1}
		Let \(r_1(x),\,p_1(x),\,r(x)\) and \(p(x)\) be positive and continuous on the interval \([a,b]\), If the derivative \( z'(x) \) of a solution \( z(x) \) of the equation
		\[
		[r_1(x)  z']' + p_1(x)  z = 0
		\]
		has consecutive zeros at \( x = c_1 \) and \( x = c_2 \) (\( a \leq c_1 < c_2 \leq b \)), and if \(r_1\ge r\) and \(p_1\le p\) holds, the derivative \( y'(x) \) of a nonnull solution \( y(x) \) of the system
		\[
		\begin{aligned}
			[r(x)  y']' + p(x)  y &= 0\\
			y'(c_1) &= 0
		\end{aligned}
		\]
		will have a zero on the interval \( (c_1, c_2] \).
	\end{theorem}
	To apply this theorem, we need the following lemma:
	\begin{lemma}\label{lem:2025-09-14-2}
		Let \(n \ge 5\) and \(\varphi(x)\) be a solution of \eqref{eq:2025-09-14-1} satisfying \(\varphi(0)=1\) and \(\varphi'(0)=0\).
		Then \(\varphi'(x)\) has a zero in \((0,1)\).
	\end{lemma}
	\begin{proof}
		Taking derivatives of \eqref{eq:2025-09-14-1} we have the following second order differential equations:
		\begin{align}
			(1-x^2)w'''-(n+3)xw''+(3n-5)w'&=0,\nonumber\\
			(1 - x^2)w^{(4)} - (n + 5)xw''' + (2n - 8)w'' &= 0,\nonumber\\
			(1 - x^2)w^{(5)} - (n + 7)xw^{(4)} + (n - 13)w''' &= 0,\nonumber\\
			(1 - x^2)w^{(6)} - (n + 9)xw^{(5)} - 20w^{(4)} &= 0,\nonumber\\
			(1 - x^2)w^{(7)} - (n + 11)x w^{(6)} - (n + 29)w^{(5)} &= 0,\nonumber \\
			(1 - x^2)w^{(8)} - (n + 13)x w^{(7)} - 2(n + 20)w^{(6)} &= 0,\nonumber \\
			(1 - x^2)w^{(9)} - (n + 15)x w^{(8)} - (3n + 53)w^{(7)} &= 0,\nonumber\\			
			(1 - x^2)w^{(10)} - (n + 17)x w^{(9)} - 4(n + 17)w^{(8)} &= 0.\label{eq:2025-09-15-1}
		\end{align}
		Combining with \(\varphi(0)=1\) and \(\varphi'(0)=0\) one gets that
		\[\varphi'''(0)=\varphi^{(5)}=\varphi^{(7)}=\varphi^{(9)}=0,\quad  \varphi''(0)=-4(n-1),\quad \varphi^{(4)}(0)=8(n-4)(n-1),\]
		\[\varphi^{(6)}(0)=160(n-4)(n-1),\quad \varphi^{(8)}(0) = 320(n + 20)(n - 4)(n - 1),\]
		\[\varphi^{(10)}(0) = 1280 (n + 17)(n + 20)(n - 4)(n - 1).\]
		Then an elementary analysis on \eqref{eq:2025-09-15-1} gives that
		\[\varphi^{(9)}(x)>0\] for \(x\in (0,1)\), from which one gets that
		\[\varphi'(x)> 4(n-1)x \left[ -1 + \frac{1}{3}(n-4)x^2 + \frac{1}{3}(n-4)x^4 + \frac{1}{63}(n+20)(n-4)x^6 \right]\]
		for \(x\in (0,1)\).
		This proves the lemma for \(n \ge 5\).
	\end{proof}
	Since the solution of \eqref{eq:2025-09-14-1} satisfying \(\frac{\dd w}{\dd x}(0)=0\) is even, the lemma above shows that the derivative of the solution to \eqref{eq:2025-09-15-2} satisfying \(w(\frac{\pi}{2})=1,\,\frac{\dd w}{\dd r}(\frac{\pi}{2})=0\) has a zero on the interval \((\frac{\pi}{2},\pi)\).
	Now we apply theorem \ref{thm:2025-09-14-1} to \eqref{eq:2025-09-15-3} and \eqref{eq:2025-09-15-2} to obtain:
	\begin{lemma}\label{thm:2025-09-15-1}
		The derivative \(\frac{\dd w}{\dd r}\) of the solution to \eqref{eq:2025-09-15-3} satisfying \(w(\frac{\pi}{2})=1,\,\frac{\dd w}{\dd r}(\frac{\pi}{2})=0\) has a consecutive zero \(c\) on the interval \((\frac{\pi}{2},\pi)\).
		Moreover, \(\frac{\dd^2 w }{\dd r^2}(c)>0\).
	\end{lemma}
	This immediately gives:
	\begin{lemma}\label{lem:2025-10-04-1}
		Let \(H=0\), \(g=4\) and \(\varepsilon>0\) be small enough, then \(\vartheta^*-\varepsilon\) is not type \(1\).
	\end{lemma}
	\begin{proof}
		If we can prove that for \(\varepsilon>0\) small enough there is a \(T(\varepsilon)\) such that \(\vartheta_\varepsilon'(T)<0\) and \(\xi_\varepsilon'(t)>0\) on \([0,T]\), then the lemma follows.
		To this end, we turn to system \eqref{eq:2025-10-04-1}, where the corresponding conclusion follows from lemma~\ref{thm:2025-09-15-1}.
		This completes the proof.	
	\end{proof}
	For \(g=1,2,3,6\), we employ a proof by contradiction.
	So we assume that the initial height \(\vartheta^*-\varepsilon\) is of type 1 for the remainder of this section.
	Note that the following proof also applies to the cases \(g=4\).
	\begin{lemma}\label{lem:2025-12-04-1}
		Let \(H=0\). If the initial height \(\vartheta^*-\varepsilon\) is type 1, then there is a unique \(T(\varepsilon)>0\) and  a unique \(T_2(\varepsilon)<T\) such that \(\xi(T)=0\) and \(\xi'(T_2)=0\), while \(\xi'(t)>0\), \(\vartheta'(t)>0\) for all \(t\in(0,T_2)\) and \(\xi'(t)<0\), \(\vartheta'(t)>0\) for all \(t\in(T_2,T)\).
	\end{lemma}
	\begin{proof}
		Assuming that the initial height \(\vartheta^*-\varepsilon\) is of type 1. Then by definition one sees that there is a unique \(T>0\) for which \(\xi(T)=0\) and \(\xi(t)\ne0\), \(\vartheta'(t)\ne 0\) for all \(t\in(0,T)\).
		Since \(\vartheta^*-\varepsilon<\vartheta^*\), one finds that \(\xi'(t)>0\) and \(\vartheta'(t)>0\) for small \(t>0\).
		This gives that \(\vartheta'(t)>0\) for all \(t\in(0,T)\).
		Moreover, \(\xi(t)\) must go through an extremum before it can go back to \(0\) and there is a \(T_2<T\) so that \(\xi'(T_2)=0\).
		The uniqueness of \(T_2\), which gives that \(\xi'(t)>0\) for all \(t\in(0,T_2)\) and that \(\xi'(t)<0\) for all \(t\in(T_2,T)\), follows by lemma~\ref{lem:2025-10-03-1}.
	\end{proof}

	\begin{lemma}\label{lem:2025-12-04-2}
		Let \(H=0\). Given \(c>0\) be small enough and \(\vartheta(T_2)<\vartheta^*+c\), where \(T_2\) is as lemma~\ref{lem:2025-12-04-1}.
		Then for \(\varepsilon\) small enough, there is a unique \(T_1<T_2\) such that \(\vartheta'(T_1)/\xi'(T_1)=1\) and \(\xi(T_1)\to\infty\) as \(\varepsilon\to0\).
	\end{lemma}
	\begin{proof}
		Let us first note that \(\vartheta(t)<\vartheta(T_2)\) for \(t\in(0,T_2)\), and that
		\[
		m_1 \cot \vartheta(T_2)-m_1 \tan \vartheta(T_2)>-\frac{m}{2}
		\]
		for \(c>0\) small enough and \(\vartheta(T_2)<\vartheta^*+c\).
		The existence of \(T_1\) follows by differential mean value theorem.
		Since the initial condition \(\varepsilon=0\) has as solution the line \((t \sin 2 \vartheta^*,\vartheta^*,0)\),
		one finds that as \(\varepsilon\to0\) the solution and its derivatives converge uniformly on compacta to the above curve.
		In particular, for on any finite interval \([0,T]\), one can make \(\vartheta'(t)/\xi'(t)\) arbitrarily small for all \(t\in[0,T]\)(by taking \(\varepsilon\) small), while \(\xi(T)\) is arbitrarily large (by taking \(T\) large and then \(\varepsilon\) small).
		This means that as \(\varepsilon\to0\) one must have \(\xi(T_1)\to\infty\), where \(T_1\) is any of the times satisfying the conditions in the lemma.
		Looking, however, at the graph ODE
		\[
		\frac{\dd^2 \vartheta}{\dd \xi^2}=\left( 1+\left( \frac{\dd \vartheta}{\dd \xi} \right)^2 \right)\left( m \frac{\dd \vartheta}{\dd \xi}\tanh \frac{2}{g}\xi+m_1 \cot \vartheta- m_2 \tan \vartheta \right),
		\]
		it follows that if \(\frac{\dd \vartheta}{\dd \xi} \ge 1\), \(\xi\) is large enough and \(\vartheta< \vartheta^*+c\) then \(\frac{\dd^2 \vartheta}{\dd \xi^2}>0\), since \(\tanh \frac{2}{g}\xi\to1\) as \(\xi\to\infty\).
		This gives the uniqueness of \(T_1\).
	\end{proof}
	These two lemma allows us to introduce the following notation:
	\[\xi_1:=\xi(T_1)\,(\text{when \(T_1\) is unique})\qquad \xi_2:=\xi(T_2),\qquad \vartheta_2:=\vartheta(T_2).\]

	\begin{lemma}\label{lem:2025-12-04-3}
		Let \(H=0\), \(T_1<T_2\) be any of the times satisfying \(\vartheta'(T_1)/\xi'(T_1)=1\), then for \(\varepsilon\) small enough one has that \(\vartheta(T_1) \ge \vartheta^*\).
	\end{lemma}
	\begin{proof}
		Let \(\xi_1=\xi(T_1)\) for simplicity, we note that \(\lim_{\varepsilon\to 0} \xi_1=\infty\) as in the proof of lemma~\ref{lem:2025-12-04-2}.
		We assume that \(\vartheta(T_1)< \vartheta^*\) for contradiction.
		Since \(\vartheta'(t) \ge 0\) for all \(t\in(0,T_2)\), we obtain that \(\vartheta(t)<\vartheta^*\) for all \(t\in[0,T_1]\).
		Then from the graph ODE
		\[
		\frac{\dd^2 \vartheta}{\dd \xi^2}=\left( 1+\left( \frac{\dd \vartheta}{\dd \xi} \right)^2 \right)\left( m \frac{\dd \vartheta}{\dd \xi}\tanh \frac{2}{g}\xi+m_1 \cot \vartheta- m_2 \tan \vartheta \right),
		\]
		one gets that \(\frac{\dd^2 \vartheta}{\dd \xi^2}>0\) for \(\xi\in[0,\xi_1]\).
		This gives that
		\[\vartheta(\frac{\xi_1}{2})<\vartheta^*-\frac{\varepsilon}{2},\qquad \frac{\dd \vartheta}{\dd \xi}(\xi_1-1)<\varepsilon\]
		and then for \(\varepsilon\) small enough:
		\[m_1 \cot \vartheta-m_2 \tan \vartheta>a_1 \varepsilon:= \frac{m_1+m_2}{2} \varepsilon\]
		provided \(\xi\in[0,\xi_1/2]\).
		Hence
		\[\frac{\dd^2 \vartheta}{\dd \xi^2}> a_1 \varepsilon\text{  for  }\xi\in[0,\xi_1/2],\]
		which implies
		\[\frac{\dd \vartheta}{\dd \xi}(1)>a_1 \varepsilon\qquad\Rightarrow\qquad \frac{\dd \vartheta}{\dd \xi} > a_1 \varepsilon\text{ for }\xi\in[1,\xi_1].\]
		Rewriting the graph ODE as
		\begin{equation}\label{eq:2025-12-04-1}
			\frac{\frac{\dd^2 \vartheta}{\dd \xi^2}}{\frac{\dd \vartheta}{\dd \xi}\left( 1+\left( \frac{\dd \vartheta}{\dd \xi} \right)^2 \right)}=m \tanh \frac{2}{g}\xi+ \frac{m_1 \cot \vartheta-m_2 \tan \vartheta}{\frac{\dd \vartheta}{\dd \xi}},
		\end{equation}
		one notes that the second term on the right-hand side is \(O(1)\) in the interval \((\xi_1-1,\xi_1)\) and that the left-hand side has
		\[
		\ln \left( \frac{\frac{\dd \vartheta}{\dd \xi}}{\sqrt{1+\left( \frac{\dd \vartheta}{\dd \xi} \right)^2}} \right)
		\]
		as anti-derivative.
		Integrating \eqref{eq:2025-12-04-1} over \((\xi_1-1,\xi_1)\) then gives
		\[- \ln \left( \frac{\dd \vartheta}{\dd \xi}(\xi_1-1) \right)=O(1)\]
		which contradicts \(\frac{\dd \vartheta}{\dd \xi}(\xi_1-1)< \varepsilon\).
	\end{proof}

	\begin{lemma}\label{lem:2025-12-04-4}
		Let \(H=0\) and \(c\) being as lemma~\ref{lem:2025-12-04-2}.
		Then for \(\varepsilon>0\) small enough one has that
		\[
		\vartheta_2 \ge \min \left\{ \vartheta^*+ c,\ \frac{\ln 2}{2(m+|G(\vartheta^*+ c)|)}+\vartheta^*\right\},
		\]
		where \(G(\vartheta) = m_1 \cot \vartheta-m_2 \tan \vartheta\).
	\end{lemma}
	\begin{proof}
		We assume that \[
		\vartheta_2 < \min \left\{ \vartheta^*+ c,\ \frac{\ln 2}{2(m+|G(\vartheta^*+ c)|)}+\vartheta^*\right\}
		\]
		and get a contradiction.
		This assumption gives that for \(\varepsilon\) small enough:
		\[\frac{\dd \vartheta}{\dd \xi} \ge 1,\quad\forall \xi\in [\xi_1,\xi_2]\]
		by lemma~\ref{lem:2025-12-04-2}, which implies that
		\[
		\frac{\frac{\dd^2 \vartheta}{\dd \xi^2}}{\frac{\dd \vartheta}{\dd \xi}\left( 1+\left( \frac{\dd \vartheta}{\dd \xi} \right)^2 \right)} \le\left| m \tan \frac{2}{g}\xi+ \frac{m_1 \cot \vartheta-m_2 \tan \vartheta}{\frac{\dd \vartheta}{\dd \xi}} \right| \le m+|G(\vartheta^*+ c)|.
		\]
		Integrating from \(\xi_1\) to \(\xi_2\) gives that
		\[\frac{1}{2}\ln 2 \le (m+|G(\vartheta^*+ c)|)(\xi_2-\xi_1)\]
		and then from lemma~\ref{lem:2025-12-04-3}
		\[\vartheta(T_2)=\vartheta(T_1)+\int_{\xi_1}^{\xi_2}\frac{\dd \vartheta}{\dd \xi}\dd \xi \ge \vartheta^*+\xi_2-\xi_1 \ge \frac{\ln 2}{2(m+|G(\vartheta^*+ c)|)}+\vartheta^*.\]
		This contradiction completes the proof.
	\end{proof}
	\begin{lemma}\label{lem:2025-12-04-5}
		Let \(H=0\). If \(\varepsilon\) is small enough then \(\vartheta^*-\varepsilon\) is not type 1.
	\end{lemma}
	\begin{proof}
		By lemma~\ref{lem:2025-12-04-1}, the assumption that \(\vartheta^*-\varepsilon\) is type 1 leads to the trajectory \[\left\{ (\xi,\vartheta)(t)\mid t\in(T_2,T)\right\}\] being a graph of \(\vartheta\) over \(\xi\).
		This graph is determined by the graph ODE
		\[
		\frac{\dd^2 \vartheta}{\dd \xi^2}=\left( 1+\left( \frac{\dd \vartheta}{\dd \xi} \right)^2 \right)\left( m \frac{\dd \vartheta}{\dd \xi}\tanh \frac{2}{g}\xi+m_1 \cot \vartheta- m_2 \tan \vartheta \right)
		\] and the initial data \(\vartheta(\xi_2)=\vartheta_2,\,\lim_{\xi\to \xi_2^-}\frac{\dd \vartheta}{\dd \xi}=-\infty\).
		By lemma~\ref{lem:2025-12-04-4}, there is a \(b>0\) such that
		\[m_1 \cot \vartheta-m_2 \tan \vartheta <-b,\quad\forall \xi\in(0,\xi_2)\]
		for the graph considered, which gives that
		\[\frac{\frac{\dd^2 \vartheta}{\dd \xi^2}}{ 1+\left( \frac{\dd \vartheta}{\dd \xi} \right)^2 } <-b\] since \(\frac{\dd \vartheta}{\dd \xi}<0\).
		Integrating from \(\xi\) to \(\xi_2\) we get that \[\frac{\pi}{2}+\arctan \frac{\dd \vartheta}{\dd \xi}>b(\xi_2-\xi)\]
		which is impossible for \(\xi=\xi_2-\frac{\pi}{2b}\).
		we conclude the proof by noting that \(\lim_{\varepsilon\to 0}\xi_2=\infty\), similarly to the proof of lemma~\ref{lem:2025-12-04-2}.
	\end{proof}

	\section{Proof of third item of proposition \ref{prop:2025-10-03-1}}\label{sec7}
	\begin{lemma}\label{lem:2025-10-03-3}
		Let \(H=0\) and \(\delta^*\) be type \(3\), then:
		\begin{enumerate}[{\rm (i)}]
			\item \(\xi'(t)>0\) and \(\vartheta'(t)>0\) for all \(t>0\).
			\item \(\lim_{t\to\infty}\alpha(t)=\frac{\pi}{2}\) and \(\alpha'(t)>0\) for large \(t\).
			\item \(\lim_{t\to\infty}\vartheta(t)=\frac{\pi}{2}\) and \(\lim_{t\to\infty}\xi(t)\) is finite,
		\end{enumerate}
		where \(\xi(t),\vartheta(t),\alpha(t)\) the solution to \eqref{eq:2025-10-02-1} with initial conditions \((\xi,\vartheta,\alpha)(0)=(0,\delta^*,0)\).
	\end{lemma}
	\begin{proof}
		By definition~\ref{def:10-25-4}, we have \(\vartheta'(t)>0\) for all \(t>0\).
		Suppose, by contradiction, that for some \(t_0>0\) one has \(\xi'(t_0)\le 0\).
		Then there is a \(T>0\) such that \(\xi'(T)=0\) and \(\xi'(t)>0\) for \(t\in(0,T)\), which follows that \(\xi''(T)<0\) by lemma~\ref{lem:2025-10-03-1}.
		Applying lemma~\ref{lem:2025-10-02-2} we now gets a contradiction since \(\xi(t)>0\) for all \(t>0\) by definition~\ref{def:10-25-4}.
		This finishes the proof of the first item.

		From the boundedness of \(\vartheta''(t)\) and \(\vartheta(t)\), we obtain that \(\lim_{t\to\infty}\vartheta'(t)=0\) since \(\vartheta'(t)>0\) for all \(t>0\).
		Note that from lemma~\ref{lem:2025-10-02-1} there is a \(b_1>0\) such that
		\[m_1 \cos ^2 \vartheta(t) - m_2 \sin^2 \vartheta(t)<-b_1\]
		for large \(t\).
		Suppose, by contradiction, that the limit \(\lim_{t\to\infty}\alpha(t)\) does not exists or less that \(\frac{\pi}{2}\).
		Then there is a \(b_2>0\) and a  sequence \(\left\{ t_i\right\}\) increasing to infinity of time such that \(\cos \alpha(t_i)>b_2\) for all \(i\).
		From the limit of \(\vartheta'(t)\) one then gets that \(\alpha'(t_i)<-b_1b_2<0\) for large \(i\), which follows that \(\alpha'(t)<-b_1b_2\) for large \(t\) by an elementary analysis.
		This, however, is impossible since \(\alpha(t)>0\) for all \(t>0\).

		From \(\lim_{t\to\infty}\alpha(t)=\frac{\pi}{2}\) we obtain that \(\xi'(t)<\vartheta'(t)\) for large \(t\).
		Together with the boundedness of \(\vartheta(t)\) we then gets the boundedness of \(\xi(t)\), which follows that \(\lim_{t\to\infty}\xi(t)\) is finite.
		Note that \(\lim_{t\to\infty}\vartheta(t)\) is also finite.
		Therefore \(\lim_{t\to\infty}\xi'(t)=\lim_{t\to\infty}\vartheta'(t)=0\) by the boundedness of \(\xi''(t)\) and \(\vartheta''(t)\).
		Then from \eqref{eq:2025-10-02-1} one has \(\lim_{t\to\infty}\vartheta(t)=\frac{\pi}{2}\).
		This finishes the proof of the third item.

		To prove \(\alpha'(t)>0\) for large \(t\), it is more convenient to work with the system \eqref{eq:2025-10-01-2}.
		Let \(T_\infty\) be right endpoint of the maximal interval of existence for corresponding solution.
		Since \(\alpha(t)\) converges to \(\frac{\pi}{2}\) from below, we then gets a sequence \(\left\{ t_i\right\}\) increasing to \(T_\infty\) such that \(\alpha'(t_i)>0\) for all \(i\).
		Suppose, for contradiction, that \(\alpha'(t) > 0\) does not hold for all large \(t>0\). Then there exists sequence \(\{\tilde{t}_i\}\) increasing to \(T_\infty\) such that \(\alpha'(\tilde{t}_i) = 0\) for all \(i\).
		From \(\lim_{t\to T_\infty}\vartheta(t)=\frac{\pi}{2}\) and \eqref{eq:2025-10-01-2} one sees that, for \(t\) approaching \(T_\infty\), \(\alpha''(t)<0\) whenever \(\alpha'(t)=0\).
		Together with the existence of that two sequences, we gets a contradiction by a standard argument.
		This finishes the proof of the second item.
	\end{proof}

	\begin{lemma}\label{lem:2025-10-03-2}
		Let \(H=0\) then \(\delta^*\) is not type \(3\).
	\end{lemma}
	\begin{proof}

		For \(\varepsilon\in[0,\delta^*)\), we denote by \(\xi_\varepsilon(t),\vartheta_\varepsilon(t),\alpha_\varepsilon(t)\) the solution to \eqref{eq:2025-10-02-1} with initial conditions \((\xi,\vartheta,\alpha)(0)=(0,\delta^*-\varepsilon,0)\).
		Note that \(\delta^*-\varepsilon\) is type 1 by the definition of \(\delta^*\), which follows that there is a \(T_2(\varepsilon)>0\) such that \[\xi_\varepsilon(T_2)=0,\quad \xi_\varepsilon(t)>0\text{ and }\vartheta'(t)>0\text{ for }t\in(0,T_2).\]
		From Rolle's theorem and lemma~\ref{lem:2025-10-03-1}, it further implies that there is a unique \(T_1(\varepsilon)\in(0,T_2)\) such that \(\xi_\varepsilon'(T_1)=0\) and
		\[\xi_\varepsilon'(t)>0\text{ for }t\in(0,T_1),\quad\xi_\varepsilon'(t)<0\text{ for }t\in(T_1,T_2).\]
		Suppose, by contradiction, that \(\delta^*\) is type \(3\), which follows from lemma~\ref{lem:2025-10-03-3} that
		\(\lim_{t\to\infty}\vartheta_0(t)=\frac{\pi}{2}.\) and \(\xi_0'(t)>0\) for all \(t>0\).
		Let \(b\) be in the interval \((0,\lim_{t\to\infty}\xi_0(t))\) and \(M>\frac{\pi}{2b}\).
		Then there is a \(T>0\) such that \(\xi_0(T)>b\) and \(m_1 \cot \vartheta_0(T) - m_2 \tan \vartheta_0(T)<-M\).
		Using the continuous dependence of the solutions to \eqref{eq:2025-10-02-1} on the initial conditions, one has for \(\varepsilon>0\) small enough that
		\[\xi_\varepsilon(T)>b,\quad m_1 \cot \vartheta_\varepsilon(T) - m_2 \tan \vartheta_\varepsilon(T)<-M,\quad \xi'_\varepsilon(t)>0\text{ for }t\in[0,T].\]
		Hence we have for \(\varepsilon\) small enough \(T_1>T\), which follows that
		\begin{equation}\label{eq:2025-10-03-1}
			\xi_\varepsilon(T_1)>b,\quad m_1 \cot \vartheta_\varepsilon(t) - m_2 \tan \vartheta_\varepsilon(t)<-M\text{ for }t\in(T_1,T_2).
		\end{equation}
		Next we note that the trajectory \(\{(\xi_\varepsilon,\vartheta_\varepsilon)(t)\,|\,t\in(T_1,T_2)\}\) can be viewed as a graph of a strictly decreasing function \(\vartheta_\varepsilon(\xi)\) over \((0,\xi_\varepsilon(T_1))\), and the function \(\vartheta_\varepsilon(\xi)\) satisfies \eqref{eq:2025-07-26-1}.
		Together with \eqref{eq:2025-10-03-1} one sees that
		\[\frac{\frac{\dd^2 \vartheta_\varepsilon}{\dd \xi^2}}{1+\left( \frac{\dd \vartheta_\varepsilon}{\dd \xi} \right)^2} <-M.\]
		Integrating this from \(0\) to \(\xi_\varepsilon(T_1)\) gives
		\[\arctan\frac{\dd \vartheta_\varepsilon}{\dd \xi}(0)>M\, \xi_\varepsilon(T_1)-\frac{\pi}{2}>M\,b-\frac{\pi}{2}>0,\] which is impossible.
	\end{proof}

\end{document}